\documentclass[12pt]{amsart}

\usepackage{amsmath,amsthm,amssymb}
\usepackage{enumerate}
\usepackage{tikz}
\usetikzlibrary{positioning}
\usepackage{hyperref}  
\usepackage{amsfonts,amscd}
\usepackage{mathtools}
\usepackage{wasysym} 
\usepackage{graphicx}
\usepackage{psfrag}
\usepackage{datetime}
\usepackage{tikz-cd}  
\usepackage{geometry} 
\usepackage[all]{xy}

\DeclareFontFamily{U}{mathx}{\hyphenchar\font45}
\DeclareFontShape{U}{mathx}{m}{n}{
      <5> <6> <7> <8> <9> <10>
      <10.95> <12> <14.4> <17.28> <20.74> <24.88>
      mathx10
      }{}
\DeclareSymbolFont{mathx}{U}{mathx}{m}{n}
\DeclareFontSubstitution{U}{mathx}{m}{n}
\DeclareMathAccent{\widecheck}{0}{mathx}{"71}

\newtheorem{theorem}{Theorem}[section]
\newtheorem{lemma}[theorem]{Lemma}
\newtheorem{proposition}[theorem]{Proposition}
\newtheorem{corollary}[theorem]{Corollary}

\newtheorem{innercustomthm}{Theorem}
\newenvironment{customthm}[1]
  {\renewcommand\theinnercustomthm{#1}\innercustomthm}
  {\endinnercustomthm}

\theoremstyle{definition}
\newtheorem{definition}[theorem]{Definition}

\newtheorem{example}[theorem]{Example}

\theoremstyle{remark}

\newcommand{\jjm}[1]{\textcolor{black}{#1}}   

\newcommand{\nw}[1]{\textcolor{black}{#1}}    
\newcommand{\jdm}[1]{\textcolor{black}{#1}}   
\newcommand{\nnw}[1]{\textcolor{black}{#1}}    

\title[]{Boundary Idempotents and $2$-precluster-tilting categories}
\author{Jordan McMahon}
\address{unaffiliated}
\email{jordanmcmahon37@gmail.com}

\begin{document}
\begin{abstract}
The homological theory of Auslander--Platzeck--Todorov on idempotent ideals laid much of the groundwork for higher Auslander--Reiten theory, providing the key technical lemmas for both higher Auslander correspondence as well as the construction of higher Nakayama algebras, among other results. Given a finite-dimensional algebra $A$ and idempotent $e$, we expand on a criterion of Jasso-K\"ulshammer in order to determine a correspondence between the $2$-precluster-tilting subcategories of $\mathrm{mod}(A)$ and $\mathrm{mod}(A/\langle e\rangle)$. This is then applied in the context of generalising dimer algebras on surfaces with boundary idempotent.
\end{abstract}

\maketitle

\section{Introduction}

Higher cluster-tilting subcategories were introduced and studied by Iyama \cite{iy3,iy1}. They remain only partially understood, and not easy to find in general. A generalisation to this are higher precluster-tilting subcategories, as introduced by Iyama and Solberg \cite{is}. These higher precluster-tilting subcategories are a weaker version of higher cluster-tilting subcategories, but which are of interest in their own right. 

A particular class of algebras where we would expect to find higher (pre)cluster-tilting subcategories are from a higher versions of dimer algebras on a disc the sense of \cite{bkm}. In recent work \cite{mw}, we were able to find higher precluster-tilting subcategories, as long as the boundary was omitted. In this work we give an inductive criterion to construct higher precluster-tilting subcategories, and show how this can be applied to boundary idempotents. This criterion is motivated by the following example.

\begin{example}
Given a semisimple algebra $A$ with two vertices $1$ and $3$, we may add a vertex $2$ together with arrows and relations, to produce a (pre)cluster-tilting subcategory in the following ways: we can form an Auslander algebra of type $A_2$:
$$
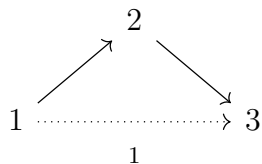
\begin{tikzcd}\ &2\arrow{dr}\\
1\arrow{ur}\arrow[dotted]{rr}& &3\end{tikzcd}$$
which has a cluster-tilting subcategory given by $S_1,S_3$ and the projective-injectives. Alternatively, we can form the (self-injective) preprojective algebra $\Pi(A_3)$, obtained from $A_3$ be adding arrows $\overline{\alpha}:j\rightarrow i$ for $\alpha:i\rightarrow j\in Q$ with admissible ideal $I$ generated by $\sum_{\alpha\in Q_1}\alpha\overline{\alpha}- \overline{\alpha}\alpha.$

$$\begin{tikzcd} 1\arrow[shift left]{r}&2\arrow[shift left]{l}\arrow[shift left]{r}&3\arrow[shift left]{l}\end{tikzcd}$$
which has a precluster-tilting subcategory given by $S_1,S_3$ and the projective-injectives.
\end{example}

This inductive criterion is based on a result of Jasso--Külshammer (Lemma \ref{jku}), who needed an inductive approach to define higher Nakayama algebras, and show that they have higher cluster-tilting subcategories.

\begin{theorem}\label{elso}
Let $A$ be a finite-dimensional algebra and $e$ an idempotent of $A$ and $\tilde{\mathcal{C}}\subseteq \mathrm{mod}(A)$ a $2$-precluster-tilting subcategory. Let $\mathcal{C}\cong \tilde{\mathcal{C}}\cap\mathrm{mod}(A/\langle e\rangle)$.
If also
\begin{enumerate}[(i)]
\item $\mathrm{Hom}_A(A/\langle e\rangle,\tilde{\mathcal{C}})\subseteq\tilde{\mathcal{C}}$.
\item $A/\langle e\rangle\otimes_A \tilde{\mathcal{C}} \subseteq\tilde{\mathcal{C}}$.
\end{enumerate}
Then $\mathcal{C}\subseteq \mathrm{mod}(A/\langle e\rangle)$ is $2$-precluster tilting.
\end{theorem}
Conversely we may (inductively) construct 2-precluster-tilting subcategories.

\begin{theorem}\label{masod}Let $A$ be a finite-dimensional algebra, $e$ an idempotent of $A$ and $\mathcal{C}\subseteq \mathrm{mod}(A/\langle e \rangle)$ a $2$-precluster-tilting subcategory.
If also
\begin{enumerate}[(i)]
\item $\mathrm{Ext}^1_A(DA,A)=0$.
\item $Ae, D(eA)$ are projective-injective $A$-modules.
\item There is an equality of sets  $$\{X\in\mathcal{C}|\mathrm{Ext}^2_A(X,J)\ne0\ \forall J\in\mathrm{inj}(A/\langle1- e \rangle)\}=\{(\tau_2^-)_AP|P\in\mathrm{proj}(A)\setminus \mathrm{proj}(A/\langle e \rangle)\}.$$
\item There is an equality of sets    $$\{X\in\mathcal{C}|\mathrm{Ext}^2_A(A/\langle 1-e \rangle,X)\ne0\}=\{(\tau_2)_AI|I\in\mathrm{inj}(A)\setminus \mathrm{inj}(A/\langle e \rangle)\}.$$
\end{enumerate}
Then $\mathcal{C}\cup\mathrm{proj}(A)\cup\mathrm{inj}(A)=:\tilde{C}\subseteq \mathrm{mod}(A)$ is $2$-precluster tilting.
\end{theorem}

The final two conditions have a combinatorial meaning in terms of the relations in the algebra. 
For a $2$-(internally)-Calabi Yau algebras, for example the preprojective algebra above, we expect significant simplifications of the above conditions. Likewise if $\mathrm{proj.dim}(S)=\mathrm{inj.dim}(S)=1$, which is often the case for a simple module over a higher Nakayama algebra \cite{fab}.

\section{Background and Notation}
Consider a finite-dimensional algebra $\nnw{A}$ over a field $K$, and fix a positive integer $d$. We will assume that $\nnw{A}$ is of the form $KQ/I$, where $KQ$ is the path algebra over some quiver $Q$ and $I$ is an admissible ideal of $KQ$. For two arrows in $Q$, $\alpha:i\rightarrow j$ and $\beta:j\rightarrow k$, we denote their composition as $\beta\alpha:i\rightarrow k$. Let $\nnw{A}^\mathrm{op}$ denote the opposite algebra of $\nnw{A}$.
An $\nnw{A}$-module will mean a finitely-generated left $\nnw{A}$-module; by $\mathrm{mod}(\nnw{A})$ we denote the category of $\nnw{A}$-modules. The functor $D=\mathrm{Hom}_K(-,K)$ defines a duality; by $\otimes$ we mean $\otimes_K$ and we denote the syzygy by $\Omega$. \jjm{Denote by $\nu:= D\nnw{A}\otimes_{\nnw{A}}-\cong D\mathrm{Hom}_{\nnw{A}}(-,\nnw{A})$ the Nakayama functor in $\mathrm{mod}(\nnw{A})$}. Let $\mathrm{add}(M)$ be the full subcategory of $\mathrm{mod}(\nnw{A})$ composed of all $\nnw{A}$-modules isomorphic to direct summands of finite direct sums of copies of $M$.

\subsection{\jdm{Higher} precluster-tilting subcategories}\label{precl-sect}

\begin{definition}\cite[Definition 2.2]{iy1}
For a finite-dimensional algebra $\nnw{A}$, a module $M\in\mathrm{mod}(\nnw{A})$ is a \emph{$d$-cluster-tilting module} if it satisfies the following conditions:
\begin{align*}
\mathrm{add}(M)&=\{X\in\mathrm{mod}(\nnw{A})\mid \mathrm{Ext}^i_{\nnw{A}}(M,X)=0\ \forall\ 0< i<d \}.\\
\mathrm{add}(M)&=\{X\in\mathrm{mod}(\nnw{A}) \mid \mathrm{Ext}^i_{\nnw{A}}(X,M)=0\ \forall\ 0< i<d \}.
\end{align*} 
In this case $\mathrm{add}(M)$ is a \emph{$d$-cluster-tilting subcategory} of $\mathrm{mod}(\nnw{A})$.
\end{definition}
Define  $\tau_d:=\tau\Omega^{d-1}$ to be the \emph{$d$-Auslander--Reiten translation} and  $\tau_{d}^-:=\nw{\tau^{-}\Omega^{-(d-1)}}$ to be the \emph{inverse $d$-Auslander--Reiten translation}.



\begin{definition}\jdm{\cite[Definition 3.2]{is}}
For a finite-dimensional algebra $\nnw{A}$, a module $M\in\mathrm{mod}(\nnw{A})$ is \emph{$d$-precluster tilting} if it satisfies the following conditions:
\begin{enumerate}[(P1)]
\item The module $M$ is a generator-cogenerator for $\mathrm{mod}(\nnw{A})$. \label{precli}
\item \nw{We have $\tau_dM \in \mathrm{add}(M)$ and $\tau^-_dM\in \mathrm{add}(M)$.}\label{preclii}
\item There is an equality $\mathrm{Ext}^i_{\nnw{A}}(M,M)=0$ for all $0<i<d$.\label{precliii}
\end{enumerate}
\end{definition}
 For a $d$-precluster-tilting module $M$, the subcategory $\mathrm{add}(M)\subseteq \mathrm{mod}(\nnw{A})$ is called a \emph{$d$-precluster-tilting subcategory}.

\begin{proposition}\cite[Theorem 1.5]{iy1}\label{lada}
We have the following 
\begin{itemize}
\item If $\mathrm{Ext}^i_A(M,A)=0$ for all $0<i<d$, then $\mathrm{Ext}^i_A(M,N)\cong D\mathrm{Ext}^{d-i}_A (N,\tau_dM)$ for all $M\in\mathrm{mod}(A)$ and all $0<i<d$.
\item If $\mathrm{Ext}^i_A(DA,N)=0$ for all $0<i<d$, then $\mathrm{Ext}^i_A(M,N)\cong D\mathrm{Ext}^{d-i}_A (\tau_d^-N,M)$ for all $N\in\mathrm{mod}(A)$ for all $0<i<d$.
\end{itemize}
\end{proposition}

\subsection{Homological theory of idempotent ideals}
Now we review the some homological theory of idempotent ideals, as introduced by Auslander, Platzeck and Todorov. Throughout this section we will let $F:=\mathrm{Hom}_A(A/\langle e\rangle,-)$.

\begin{proposition}\cite[Proposition 1.1]{apt}\label{apt11}
Let $N$ be an $A$-module, and let $1\leq d \leq \infty$. Then the following are equivalent:
\begin{enumerate}[(i)]
\item $\mathrm{Ext}^i_A(A/\langle e \rangle, N)=0$ for all $i$ such that $0<i<d$.
\item Let $M$ be in $\mathrm{mod}(A/\langle e \rangle)$. Then there are isomorphisms: $$\mathrm{Ext}^i_{A/\langle e\rangle}(M,FN))\rightarrow \mathrm{Ext}^i_A(M,N)$$  for all $0<i<d$.
\end{enumerate}
\end{proposition}

A third equivalent condition was incorrectly stated in the original article. The result we will need instead is the following:

\begin{corollary}\label{mcapt}
Let $N$ be an $A$-module, and let $0\rightarrow N \rightarrow I_0\rightarrow I_1\rightarrow \cdots\rightarrow I_d$ be the beginning of a minimal injective coresolution of $N$ and let $0<i<d$. Then each equivalent condition of Proposition \ref{apt11} implies 
$$0\rightarrow FN \rightarrow FI_0\rightarrow \cdots \rightarrow FI_d$$ is the beginning of an injective coresolution of $FN$ in $\mathrm{mod}(A/\langle e\rangle)$.
\end{corollary}
\begin{proof}
Suppose that $\mathrm{Ext}^i_A(A/\langle e \rangle, N)=0$ for all $i$ such that $0<i<d$, and let $C_j:=\mathrm{coker}(I_{j-1}\rightarrow I_{j})$ . Then we have an exact sequence:
$$0\rightarrow FN\rightarrow\cdots \rightarrow FI_{d-2}\rightarrow FC_{d-1}\rightarrow 0$$ since  $\mathrm{Ext}^i_A(A/\langle e \rangle, N)=0$ for all $0<i<d$. Moreover, the exact sequences
\begin{align*}
0&\rightarrow FC_{d-1} \rightarrow FI_{d-1}\rightarrow FC_{d}\\
0&\rightarrow FC_{d}\rightarrow FI_{d}
\end{align*}
combine to give the result.
\end{proof}
We note that the resulting injective coresolution is not necessarily minimal.
There is now a characterisation
 \begin{proposition}\cite[Proposition 1.3]{apt}\label{aptii}
Let $A$ be a finite-dimensional algebra and $e$ an idempotent of $A$. Then the following are equivalent
\begin{enumerate}[(i)]
\item There are isomorphsims  $\mathrm{Ext}^i_{A/\langle e\rangle}(M, N)\rightarrow \mathrm{Ext}^i_A(M,N)$ for all $M,N\in \mathrm{mod}(A/\langle e\rangle)$ and all $0\leq i <d$. 
\item $\mathrm{Ext}^i_A(A/\langle e \rangle, N)=0$ for all $N\in\mathrm{mod}(A/\langle e\rangle)$ for all $i$ such that $0<i<d$.
\item $\mathrm{Ext}^i_A(A/\langle e \rangle, I)=0$ for all $I\in\mathrm{inj}(A/\langle e\rangle)$ for all $i$ such that $0<i<d$.
\end{enumerate}
\end{proposition}
In this case, the ideal $\langle e \rangle$ is said to be \emph{$(d-1)$-idempotent}.
A related useful result is the following. For a positive integer $d$, we define $\mathbf{I}_{d}$ to be the full subcategory of $\mathrm{mod}(A)$ consisting of the $A$-modules $M$ having an injective resolution
$$0 \rightarrow M\rightarrow I_0\rightarrow I_1 \rightarrow \cdots $$ with $I_j\in\mathrm{add}(I)$ for all $0\leq i\leq d$.

 \begin{proposition}\cite[Proposition 2.6]{apt}\label{tooapt}
Let $A$ be a finite-dimensional algebra, $e$ an idempotent of $A$ and $I=D(1-e)A$ and $1\leq d<\infty$. Then the following are equivalent
\begin{enumerate}[(i)]
\item $N\in\mathbf{I}_d$. 
\item $\mathrm{Ext}^i_A(M, N)=0$ for all $M\in\mathrm{mod}(A/\langle e\rangle)$ for all $i$ such that $0\leq i<d$.
\item $\mathrm{Ext}^i_A(A/\langle e \rangle, N)=0$ for all $0\leq i<d$.
\end{enumerate}
\end{proposition}

\subsection{Main results}
\begin{customthm}{\ref{elso}}
Let $A$ be a finite-dimensional algebra and $e$ an idempotent of $A$ and $\tilde{\mathcal{C}}\subseteq \mathrm{mod}(A)$ a $2$-precluster-tilting subcategory.
If also
\begin{enumerate}[(i)]
\item $\mathrm{Hom}_A(A/\langle e\rangle,\tilde{\mathcal{C}})\subseteq\tilde{\mathcal{C}}$
\item $A/\langle e\rangle\otimes_A \tilde{\mathcal{C}} \subseteq\tilde{\mathcal{C}}$.
\end{enumerate}
Then $\mathcal{C}:=\tilde{\mathcal{C}}\cap\mathrm{mod}(A/\langle e \rangle)\subseteq \mathrm{mod}(A/\langle e\rangle)$ is $2$-precluster tilting.
\end{customthm}

\begin{proof}
Suppose $\tilde{\mathcal{C}} \subseteq \mathrm{mod}(A)$ is $2$-precluster-tilting. By assumption (i), we have $\mathrm{inj}(A/\langle e\rangle) \in \mathcal{C}$ and by assumption (ii) $\mathrm{proj}(A/\langle e\rangle)\in \mathcal{C}.$ So condition (P\ref{precli}) is satisfied. Secondly, Proposition \ref{aptii}(iii) implies that $\langle e\rangle$ is $1$-idempotent, and hence that $\mathrm{Ext}^1_{A/\langle e\rangle}(M,N)\cong \mathrm{Ext}^1_A(M,N)=0$ for all $M,N\in\mathcal{C}$. Hence condition (P\ref{precliii}) is also satisfied.

Finally, let $N\in\mathcal{C}$ be a non-injective $A$-module and let $$0\rightarrow N \rightarrow \tilde{I_0}\rightarrow \tilde{I_1}\rightarrow \tilde{I_2}$$ be the beginning of a minimal injective coresolution of $N$ in $\mathrm{mod}(A)$. It follows that $(\tau_2^-)_AN= \mathrm{coker}(\tilde{P}_{1}\rightarrow \tilde{P}_2)$. Let $I_j=\mathrm{Hom}_A(A/\langle e\rangle,\tilde{I_j})$. Then Corollary \ref{mcapt} implies that $0\rightarrow N \rightarrow I_0\rightarrow I_1\rightarrow  I_2$ is the beginning of an injective coresolution of $N$ in $\mathrm{mod}(A/\langle e \rangle)$. Since $I_0$ is necessarily minimal, the only case where non-minimality may arise is a trivial map to a summand of $I_2$. It follows $(\tau_2^-)_{A/\langle e \rangle}  N$ is a summand of $\mathrm{coker}(P_{1}\rightarrow P_2)=\mathrm{Hom}_A(A/\langle e\rangle, (\tau^-_2)_AN) \in \mathcal{C}$ by assumption. Hence $(\tau^-_2)_AN\in\mathcal{C}$ and dually $\mathcal{C}$ is closed under $(\tau_2)_A$. So condition (P\ref{preclii}) holds, and $\mathcal{C}\subseteq\mathrm{mod}(A/\langle e\rangle)$ is a $2$-precluster-tilting subcategory.
\end{proof}

We need a technical result:
\begin{lemma}
Let $A$ be a finite-dimensional algebra such that  
$Ae, D(eA)\in\mathrm{proj}$-$\mathrm{inj.}(A)$.

Then for any $M\in\mathrm{mod}(A)$ with minimal injective resolution $$0\rightarrow M\rightarrow I_0\rightarrow I_1\rightarrow I_2,$$ then $$I_2\in\mathrm{add}(D(eA)) \iff \mathrm{Ext}_A^2(A/\langle 1-e \rangle ,M)=0.$$
\end{lemma}
\begin{proof}
First note
$\mathrm{Ext}^2_A(A/\langle 1-e \rangle,M)\cong \overline{\mathrm{Hom}}_A(A/\langle 1-e \rangle,\Omega^{-2}(M))$. 
Now any morphism $A/\langle 1-e \rangle\rightarrow \Omega^{-2}(M)$ that factors through an injective must factor through an injective summand of $\Omega^{-2}(M)$. This is impossible, since any such summand is in $\mathrm{add}(D(eA))$ and therefore projective by assumption. Hence $\overline{ \mathrm{Hom}}_A(A/\langle 1-e \rangle,\Omega^{-2}(A))\cong\mathrm{Hom}_A(A/\langle 1-e \rangle,\Omega^{-2}(A)).$ So $\mathrm{Ext}^2_A(A/\langle1- e \rangle,M)\cong \overline{\mathrm{Hom}}_A(A/\langle1- e \rangle,\Omega^{-2}(M))\cong\mathrm{Hom}_A(A/\langle1- e \rangle,\Omega^{-2}(M))$  and the result follows from Proposition \ref{tooapt}.
\end{proof}
\begin{customthm}{\ref{masod}}Let $A$ be a finite-dimensional algebra, $e$ an idempotent of $A$ and $\mathcal{C}\subseteq \mathrm{mod}(A/\langle e \rangle)$ a $2$-precluster-tilting subcategory.
If also
\begin{enumerate}[(i)]
\item $\mathrm{Ext}^1_A(DA,A)=0$.
\item $Ae, D(eA)\in\mathrm{proj}$-$\mathrm{inj.}(A)$.
\item There is an equality of sets  $$\{X\in\mathcal{C}|\mathrm{Ext}^2_A(X,J)\ne0\ \forall J\in\mathrm{inj}(A/\langle1- e \rangle)\}=\{(\tau_2^-)_AP|P\in\mathrm{proj}(A)\setminus \mathrm{proj}(A/\langle e \rangle)\}.$$
\item There is an equality of sets    $$\{X\in\mathcal{C}|\mathrm{Ext}^2_A(A/\langle 1-e \rangle,X)\ne0\}=\{(\tau_2)_AI|I\in\mathrm{inj}(A)\setminus \mathrm{inj}(A/\langle e \rangle)\}.$$
\end{enumerate}
Then $\mathcal{C}\cup\mathrm{proj}(A)\cup\mathrm{inj}(A)=:\tilde{C}\subseteq \mathrm{mod}(A)$ is $2$-precluster tilting.
\end{customthm}

\begin{proof}

Suppose that $\mathcal{C}\subseteq\mathrm{mod}(A/\langle e\rangle)$ is $2$-precluster-tilting. We have that $\tilde{\mathcal{C}}$ is a generator-cogenerator by construction. Since $\mathrm{Ext}_A^1(DA,A)=0$, Proposition \ref{lada} implies for any $N\in \mathcal{C}\cup\ \mathrm{proj}(A)$ the calculation $$\mathrm{Ext}^1_A(DA, N) \cong D\mathrm{Ext}^{1}(N,(\tau_2)_ADA))\cong 0,$$ since $(\tau_2)_A(DA)\in \mathcal{C}$. Dually $\mathrm{Ext}^1_A(M,A)=0$ for all $M\in\mathcal{C}$. By construction $\mathrm{Ext}_A^1(P,I)=0$ for any $P\in\mathrm{proj}(A/\langle e \rangle)$ and $I\in\mathrm{inj}(A/\langle e \rangle)$ (there are no arrows in the quiver of $A$ from a sink in the quiver of $A/\langle e\rangle$ to a source in the quiver of $A/\langle e\rangle$). So Proposition \ref{aptii} implies $0=\mathrm{Ext}^1_{A/\langle e \rangle}(M,N)=\mathrm{Ext}_A^1(M,N)$ for all $M,N\in\mathcal{C}$. Hence condition (P\ref{precliii}) is also satisfied.

Finally, we have to show closure under $(\tau_2^-)_A$, we do this for a given $X\in\mathcal{C}$ with minimal injective resolution in $\mathrm{mod}(A/\langle e\rangle)$: $0\rightarrow X\rightarrow I_0\rightarrow I_1\rightarrow I_2$, where we set $\tilde{J},\tilde{J}^\prime$ to be injective $A$-modules such that $\mathrm{Hom}_A(A/\langle e \rangle, \tilde{J})=0$.
\begin{enumerate}[(a)]
\item  $0\rightarrow X\rightarrow \tilde{I}_0\rightarrow \tilde{I}_1(\oplus \tilde{J})\rightarrow \tilde{I}_2  \oplus \tilde{J}^\prime$ is a minimal injective resolution of $X$ in $\mathrm{mod}(A)$: then by assumption $(\tau_d^-)_AX\in\mathrm{inj}(A)\in\tilde{\mathcal{C}}$.
\item  $0\rightarrow X\rightarrow \tilde{I}_0\rightarrow \tilde{I}_1(\oplus \tilde{J})\rightarrow \tilde{I}_2$ is an injective resolution of $X$ in $\mathrm{mod}(A)$: then simply $(\tau_d^-)_AX\cong (\tau_d^-)_{A/\langle e\rangle}X\in\tilde{\mathcal{C}}$.
\item  $0\rightarrow X\rightarrow \tilde{I}_0\rightarrow \tilde{J} $ is an injective resolution of $X$ in $\mathrm{mod}(A)$: then $X\in\mathrm{inj}(A/\langle e \rangle)$ and  $(\tau_d^-)_AX\in\mathrm{proj}(A/\langle e\rangle)\in\tilde{\mathcal{C}}$.
\end{enumerate}
Hence $\tilde{\mathcal{C}}$ is closed under $(\tau_d^-)_A$, and dually also under $(\tau_2)_A$. So $\tilde{\mathcal{C}}\subseteq\mathrm{mod}(A)$ is a 2-precluster-tilting subcategory.
\end{proof}

\section{Examples}

In this section we will consider algebras with vertices labelled by subsets of $\{1,2,\ldots,n\}$. There is a canonical way of constructing the algebra. Let $Q_0$ be a set of $(d+1)$-subsets of $\{1,2,\ldots,n\}$.
For $X,Y\in Q_0$, define $Q_1$ by adding arrows $\alpha_{i}(X):X\rightarrow Y$ wherever $X\setminus\{i\}=Y\setminus\{i+1\}$ for some $i\in X$. Let $I$ be the admissible ideal of $KQ$ generated by the elements $$\alpha_j(\alpha_i(I))-\alpha_i(\alpha_j(I)),$$ which range over all $X\in Q_0$. By convention, $\alpha_i(X)=0$ whenever $X$ or $X\cup\{i+1\}\setminus \{i\}$ is not a member of $Q_0$. Hence there are zero relations included in the ideal $I$.

\subsection{Higher Nakayama algebras}

One of the motivating examples comes from higher Nakayama algebras \cite{jk}. In order to define higher Nakayama algebras and define higher cluster-tilting subcategories, Jasso and K\"ulshammer make use of the following result, which motivates our main results.

\begin{lemma}\label{jku}\cite[Lemma 1.20]{jk}
Let $A$ be a finite-dimensional algebra and $\tilde{\mathcal{C}}$ a $d$-cluster-tilting subcategory of $\mathrm{mod}(A)$. Let $e\in A$ be an idempotent such that the following conditions are satisfied:
\begin{itemize}
\item All the projective and all the injective $A/\langle e \rangle$-modules belong to $\mathcal{C}$.
\item Every indecomposable $M\in\mathcal{C}$ which does not lie in $\mathrm{mod}(A/\langle e\rangle)$ is projective-injective.
\end{itemize}
Then $\langle e\rangle$ is a $(d-1)$-idempotent ideal and $\mathcal{C}\subseteq\mathrm{mod}(A/\langle e\rangle)$ is $d$-cluster tilting.
\end{lemma}
Using this result, Jasso and K\"ulshammer are able to inductively define higher cluster-tilting subcategories. For example, the following higher Nakayama algebra $A$ has quiver

$$
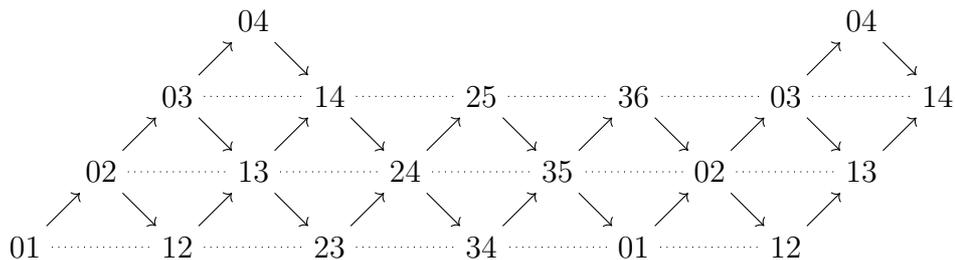
\begin{tikzpicture}
\node(a) at (0,0){$01$};
\node(b) at (1,1){$02$};
\node(c) at (2,2){$03$};
\node(d) at (3,3){$04$};
\node(e) at (2,0){$12$};
\node(f) at (3,1){$13$};
\node(g) at (4,2){$14$};
\node(h) at (4,0){$23$};
\node(i) at (5,1){$24$};
\node(j) at (6,2){$25$};
\node(k) at (6,0){$34$};
\node(l) at (7,1){$35$};
\node(m) at (8,2){$36$};
\node(n) at (8,0){$01$};
\node(o) at (9,1){$02$};
\node(p) at (10,2){$03$};
\node(q) at (11,3){$04$};
\node(r) at (10,0){$12$};
\node(s) at (11,1){$13$};
\node(t) at (12,2){$14$};

\path[->] (a) edge (b);
\path[->] (b) edge (c);
\path[->] (c) edge (d);
\path[->] (e) edge (f);
\path[->] (f) edge (g);
\path[->] (h) edge (i);
\path[->] (i) edge (j);
\path[->] (k) edge (l);
\path[->] (l) edge (m);
\path[->] (n) edge (o);
\path[->] (o) edge (p);
\path[->] (p) edge (q);
\path[->] (r) edge (s);
\path[->] (s) edge (t);

\path[->] (b) edge (e);
\path[->] (c) edge (f);
\path[->] (d) edge (g);
\path[->] (f) edge (h);
\path[->] (g) edge (i);
\path[->] (i) edge (k);
\path[->] (j) edge (l);
\path[->] (l) edge (n);
\path[->] (m) edge (o);
\path[->] (o) edge (r);
\path[->] (p) edge (s);
\path[->] (q) edge (t);

\path[-,dotted] (a) edge (e);
\path[-,dotted] (e) edge (h);
\path[-,dotted] (h) edge (k);
\path[-,dotted] (n) edge (k);
\path[-,dotted] (n) edge (r);

\path[-,dotted] (b) edge (f);
\path[-,dotted] (f) edge (i);
\path[-,dotted] (i) edge (l);
\path[-,dotted] (o) edge (l);
\path[-,dotted] (o) edge (s);

\path[-,dotted] (c) edge (g);
\path[-,dotted] (g) edge (j);
\path[-,dotted] (j) edge (m);
\path[-,dotted] (p) edge (m);
\path[-,dotted] (p) edge (t);
\end{tikzpicture}$$
and relations indictated by the dotted arrows. We may easily apply Theorems \ref{elso} and \ref{masod} to the idempotent $e_{04}$, since $S_{04}$ has projective dimension and injective dimension 1.


\subsection{Boundary idempotents}
 Scott found a cluster structure on the Grassmannian $\mathbb{C}[\mathrm{Gr}(k,n)]$, with clusters given by non-crossing $k$-subsets of $[n]$. On the other hand, Oppermann and Thomas \cite{ot} generalised the cluster structure of triangulations of convex polygons to cyclic polytopes.  Combinatorially, a triangulation of a cyclic polytope is given by maximal-by-size collections of non-intertwining subsets, where, given two $k$-subsets $I=\{i_1,i_2,\ldots,i_l\}$ and $J=\{j_1,j_2,\ldots,j_l\}$, then \emph{$I$ intertwines $J$} if \[i_{1} < j_{1} < i_{2} < \dots < i_{l} < j_{l}.\] While no cluster algebra is formed, such triangulations are related to the representation theory of higher Auslander algebras of Dynkin type $A$. 
Two $k$-subsets $I$ and $J$ are said to be \emph{non-crossing} if there do not exist elements $s<t<u<v$ (ordered cyclically) where $s,u\in I-J$, and $t,v\in J-I$. 


Oppermann and Thomas \cite{ot} were able to describe higher Auslander algebras of type $A$ by maximal collections of non-intertwining subsets, and also to triangulations of cyclic polytopes. In \cite{mw}, we extended this to tensor products of higher Auslander algebras of type $A$ by introducing maximal collections of non-$l$-intertwining subsets. Critically, we were only able to find higher precluster-tilting subcategories in general. 

Baur, King and Marsh \cite{bkm} studied dimer algebras on a disc, which are related to maximal non-crossing collections (and tensor products of type $A$ quivers). In their work, boundary idempotents (given by consecutive subsets) play a key role. The criterion in Theorem \ref{masod} gives us hope that we may inductively add these boundary idempotents back in, when we consider non-intertwining and non-$l$-intertwining collections. This is illustrated in the following example:

\begin{example}
Consider the algebra $A$, it can be checked that the idempotent $e=E_{125,236,145,367,147}$ satisfies the conditions for Theorem \ref{elso} and \ref{masod}. $A/\langle e\rangle$ can be described by a maximal non-intertwining collection of $3$-subsets of $\{1,2,\ldots,6\}$ (135,136,146) as well as a semisimple algebra (256,347). So $\mathrm{mod}(A/\langle e\rangle)$ contains a $2$-cluster-tilting subcategory, hence $\mathrm{mod}(A)$ also contains a (pre)cluster-tilting subcategory. 

$$
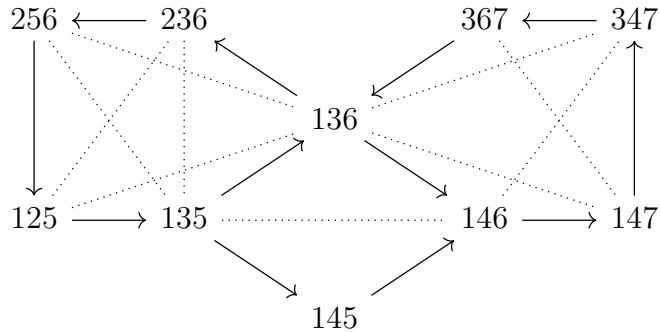
\begin{tikzcd}
256\arrow{dd}\arrow[-,dotted]{rrd}\arrow[-,dotted]{ddr}& 236\arrow{l}\arrow[-,dotted]{dd}&& 367\arrow[-,dotted]{rdd}\arrow{dl}&347\arrow{l}\\
&&136\arrow{ul}\arrow{dr}\arrow[-,dotted]{rru}\arrow[-,dotted]{rrd}\\
125\arrow[-,dotted]{rru}\arrow[-,dotted]{ruu}\arrow{r}&135\arrow{ur}\arrow[-,dotted]{rr}\arrow{dr}&&146\arrow{r}\arrow[-,dotted]{ruu}&147\arrow{uu}\\
&&145\arrow{ur}
\end{tikzcd}$$

\end{example}

\bibliographystyle{amsplain}
\bibliography{extendinging}

\providecommand{\bysame}{\leavevmode\hbox to3em{\hrulefill}\thinspace}
\providecommand{\MR}{\relax\ifhmode\unskip\space\fi MR }
\providecommand{\MRhref}[2]{%
  \href{http://www.ams.org/mathscinet-getitem?mr=#1}{#2}
}
\providecommand{\href}[2]{#2}
\begin{thebibliography}{1}

\bibitem{apt}
M.~Auslander, M.~I. Platzeck, and G.~Todorov, \emph{Homological theory of
  idempotent ideals}, Trans. Amer. Math. Soc. \textbf{332} (1992), no.~2,
  667--692. \MR{1052903}

\bibitem{bkm}
Karin Baur, Alastair~D. King, and Bethany~R. Marsh, \emph{Dimer models and
  cluster categories of {G}rassmannians}, Proc. Lond. Math. Soc. (3)
  \textbf{113} (2016), no.~2, 213--260. \MR{3534972}

\bibitem{iy3}
Osamu Iyama, \emph{Auslander correspondence}, Adv. Math. \textbf{210} (2007),
  no.~1, 51--82. \MR{2298820}

\bibitem{iy1}
\bysame, \emph{Higher-dimensional {A}uslander-{R}eiten theory on maximal
  orthogonal subcategories}, Adv. Math. \textbf{210} (2007), no.~1, 22--50.
  \MR{2298819}

\bibitem{is}
Osamu Iyama and {\O}yvind Solberg, \emph{Auslander-{G}orenstein algebras and
  precluster tilting}, Adv. Math. \textbf{326} (2018), 200--240. \MR{3758429}

\bibitem{jk}
Gustavo Jasso, Julian K{\"u}lshammer, Chrysostomos Psaroudakis, and Sondre
  Kvamme, \emph{Higher nakayama algebras {I}: construction}, Advances in
  Mathematics \textbf{351} (2019), 1139--1200.

\bibitem{fab}
Jordan McMahon, \emph{Fabric idempotents and higher {A}uslander--{R}eiten
  theory}, Journal of Pure and Applied Algebra \textbf{224} (2020), no.~8,
  106343.

\bibitem{mw}
Jordan McMahon and Nicholas~J. Williams, \emph{The combinatorics of tensor
  products of higher auslander algebras of type {A}}, Glasgow Mathematical
  Journal (2020), 1–21.

\bibitem{ot}
Steffen Oppermann and Hugh Thomas, \emph{Higher-dimensional cluster
  combinatorics and representation theory}, J. Eur. Math. Soc. (JEMS)
  \textbf{14} (2012), no.~6, 1679--1737. \MR{2984586}

\end{thebibliography}

\end{document}